\documentclass[12pt]{article}
\usepackage[a4paper]{anysize}\marginsize{2.5cm}{2.5cm}{1.5cm}{1.5cm}
\pdfpagewidth=\paperwidth \pdfpageheight=\paperheight
\usepackage{amsfonts,amssymb,amsthm,amsmath,eucal,tabu,url}
\usepackage{pgf}
\usepackage{array}
 \usepackage{pstricks}
 \usepackage{pstricks-add}
 \usepackage{pgf,tikz}
 \usetikzlibrary{automata}
 \usetikzlibrary{arrows}
 \usepackage{indentfirst}
\theoremstyle{plain}
\newtheorem{thm}{Theorem}[section]
\newtheorem{theorem}[thm]{Theorem}

\newtheorem{proposition}[thm]{Proposition}

\theoremstyle{definition}
\newtheorem{definition}[thm]{Definition}
\newtheorem{remark}[thm]{Remark}
\newtheorem{example}[thm]{Example}

\newtheorem{question}[thm]{Question}

\newtheorem{thevarthm}[thm]{\varthmname}

\newenvironment{varthm*}[1]{\trivlist\item[]{\bf #1.}\it}{\endtrivlist}

\def\keywordname{{\bfseries Keywords}}%
\def\keywords#1{\par\addvspace\medskipamount{\rightskip=0pt plus1cm
\def\and{\ifhmode\unskip\nobreak\fi\ $\cdot$
}\noindent\keywordname\enspace\ignorespaces#1\par}}
\def\subclassname{{\bfseries Mathematics Subject Classification
(2000)}\enspace}
\def\subclass#1{\par\addvspace\medskipamount{\rightskip=0pt plus1cm
\def\and{\ifhmode\unskip\nobreak\fi\ $\cdot$
}\noindent\subclassname\ignorespaces#1\par}}

\begin{document}
\title{On Seshadri constants and point-curve configurations}
\author{Marek Janasz and Piotr Pokora }
\date{\today}
\maketitle
\begin{abstract}
In the note we study the multipoint Seshadri constants of $\mathcal{O}_{\mathbb{P}^{2}_{\mathbb{C}}}(1)$ centered at singular loci of certain curve arrangements in the complex projective plane. Our first aim is to show that the values of Seshadri constants can be approximated with use of a combinatorial invariant which we call the configurational Seshadri constant. We study specific examples of point-curve configurations for which we provide actual values of the associated Seshadri constants. In particular, we provide an example based on Hesse point-conic configuration for which the associated Seshadri constant is computed by a line. This shows that multipoint Seshadri constants are not purely combinatorial.
\subclass{14C20,14N10, 14N20}
\end{abstract}
\thispagestyle{empty}

\section{Introduction}
In the present note we study multipoint Seshadri  constants of $\mathcal{O}_{\mathbb{P}^{2}_{\mathbb{C}}}(1)$ centered at singular loci of a certain class of curve arrangements in the complex projective plane. This path of studies was initiated by the second author in \cite{Pokora2019} in the context of line arrangements and special singular curves in the complex projective plane. Before we present the main goal of our paper, let us present basics on the multipoint Seshadri constants. Let
$X$ be a complex projective variety of dimension $\dim X = n$ and let $L$ be a nef line bundle. The multipoint Seshadri constant of $L$ at $r\geq 1$ points $x_{1}, ..., x_{r} \in X$ is defined as 
$$\varepsilon(X,L;\, x_{1}, ..., x_{r}) = \inf_{\{x_{1},...,x_{r}\} \cap C \neq \emptyset} \frac{L\cdot C}{\sum_{i=1}^{r} {\rm mult}_{x_{i}}C},$$
where the infimum is taken over all irreducible and reduced curves $C$ on $X$, and as usually ${\rm mult}_{x_{i}}(C)$ denotes the multiplicity of $C$ at $x_{i}$. There exists an upper-bound on the multipoint Seshadri constant, namely we have
$$\varepsilon(X,L;\, x_{1}, ..., x_{r})  \leq \sqrt[n] {\frac{L^{n}}{r} }.$$
It is well-known that the multipoint Seshadri constant of $L$, treated as a function of points $x_{1}, ..., x_{r}$ as variables, attains its maximal value at a set of generic points, see \cite{Oguiso}. 

If we restrict our attention to the case when $X = \mathbb{P}^{2}_{\mathbb{C}}$, then the Seshadri constant of $\mathcal{O}_{\mathbb{P}^{2}_{\mathbb{C}}}(1)$ centered at $r$ generic points is governed by the celebrated Nagata conjecture, so  we have, at least, a conjectural picture of what can happen in this scenario \cite{SzSz}. On the other side, we do not know much about a potential or even hypothetical behaviour of the multipoint Seshadri constants if we allow to consider special point configurations in the complex projective plane. The Leitmotif of our investigations is oriented by the following question of the second author that recently has attracted the attention of researchers \cite{Ober}. 
\begin{question}
\label{MQ}
Let $Z$ be the set of all singular points of an arrangement
of lines $\mathcal{L} \subset \mathbb{P}^{2}_{\mathbb{C}}$ which is not a pencil of lines. Is it true that the multipoint Seshadri constant is equal to
$$\varepsilon(\mathbb{P}^{2}_{\mathbb{C}}, \mathcal{O}_{\mathbb{P}^{2}_{\mathbb{C}}}(1);Z) = \frac{1}{{\rm mpl}(Z)},$$
where ${\rm mpl}(Z)$ is the maximal number of collinear points in $Z$?
\end{question}
In other words, in this question we ask whether the combinatorics of the line arrangement $\mathcal{L}$ would be enough to compute the multipoint Seshadri constants $\mathcal{O}_{\mathbb{P}^{2}_{\mathbb{C}}}(1)$ centered at the singular locus of $\mathcal{L}$. This approach sits on the boundary of the combinatorics and algebraic geometry, and it might lead to new developments in these two areas. We predict that this question should not have an affirmative answer, but somehow surprisingly it is difficult to verify it in the whole generality due to complications that occur when we study the geometry of line arrangements. On the other side, it seems natural to extend such studies to a wider class of curve arrangements, namely to the so-called $d$-arrangements of plane curves. This notion was introduced in \cite{PRSz} in the context of Harbourne indices and the bounded negativity conjecture. It turned out that it is more efficient to study the negativity phenomenon from a viewpoint of curve arrangements instead of focusing on the case of irreducible curves which are notoriously difficult to construct. Having this motivation, we decided to introduce a new version of the multipoint Seshadri constants for reduced curves.
\begin{definition}
Let $\mathcal{C} = \{C_{1}, ..., C_{k}\} \subset \mathbb{P}^{2}_{\mathbb{C}}$ be an arrangement of irreducible curves and denote by ${\rm Sing}(\mathcal{C})$ the singular locus of $\mathcal{C}$, i.e., the set of all singular points of the components and points where two or more curves intersect. We define the configurational Seshadri constant of $\mathcal{C}$ as
$$\varepsilon_{\mathcal{C}}(\mathcal{O}_{\mathbb{P}^{2}_{\mathbb{C}}}(1)) =  \frac{{\rm deg}(\mathcal{C})}{\sum_{P \in {\rm Sing}(\mathcal{C})}{\rm mult_{P}}(\mathcal{C})}.$$

\end{definition}
At the first glance this notion seems to be far away from the classical multipoint Seshadri constants, but the example below shows something opposite.
\begin{example}
Let $\mathcal{F} = \{\ell_{1}, ..., \ell_{3n}\} \subset \mathbb{P}^{2}_{\mathbb{C}}$ be the $n$-th Fermat arrangement of $3n$ lines. This arrangement is given by the zeros of the following defining polynomial
$$Q(x,y,z) :=(x^{n}-y^{n})(y^{n}-z^{n})(z^{n}-x^{n}).$$
A simple check tells us that the arrangement has exactly $n^2$ intersection points of multiplicity $3$ and exactly $3$ points of multiplicity $n$. It was shown in \cite{Pokora2019} that if $n\geq 2$ one has
$$\varepsilon( \mathbb{P}^{2}_{\mathbb{C}},\mathcal{O}_{\mathbb{P}^{2}_{\mathbb{C}}}(1); {\rm Sing}(\mathcal{F})) = \frac{1}{n+1},$$ and we also have
$$\varepsilon_{\mathcal{F}}(\mathcal{O}_{\mathbb{P}^{2}_{\mathbb{C}}}(1)) = \frac{3n}{3\cdot n^{2} + n \cdot 3} = \frac{1}{n+1},$$
so these two values coincide for every $n\geq 2$.
\end{example}
Let us present briefly the content of the present note. In Section $2$, we recall the notion of $d$-arrangements $\mathcal{C}$ of plane curves and  we provide for them a lower bound on $\varepsilon_{\mathcal{C}}(\mathcal{O}_{\mathbb{P}^{2}_{\mathbb{C}}}(1))$. Our bound has some space for improvements, but the main advantage of our approach is that the bound does not depend on the geometry of particular curve arrangements and gives the correct order of the magnitude for the expected values of $\varepsilon_{\mathcal{C}}$. In Section $3$, we provide actual values of the multipoint Seshadri constants for some classes of curve arrangements and we compare them with the associated values of $\varepsilon_{\mathcal{C}}$. Our main result in Section 3 shows that there exists a very specific arrangement of $12$ conics, called both Chilean and Hesse arrangement, such that the multipoint Seshadri constant for $\mathcal{O}_{\mathbb{P}^{2}_{\mathbb{C}}}(1)$ centered at the singular locus consisting of $21$ points is computed by a line.

\textbf{Notation:} We are working exclusively over the complex numbers. If $p$ is a point on a curve $C$, then we denote by ${\rm mult}_{p}(C) = m_{p}(C)$ the multiplicity of $C$ at $p$, and if it is clear from the context which curve is considered we abbreviate as $m_{p}$. Abusing the notation, we will consider a curve arrangement $\mathcal{C}$ both as a combinatorial object and as a divisor hoping that it will not lead to confusion.

\section{Configurational Seshadri constants for certain point-curve configurations}
In this section we will consider a special class of point-curve configurations, the so-called $d$-arrangements.

\begin{definition}
Let $\mathcal{C} = \{C_{1}, ..., C_{k}\} \subseteq \mathbb{P}^{2}_{\mathbb{C}}$ be an arrangement of $k\geq 3$ curves. Then $\mathcal{C}$ is a $d$-arrangement with $d\geq 1$ if
\begin{itemize}
    \item every $C_{i}$ is smooth of degree ${\rm deg}(C_{i}) = d$;
    \item all singular points of $\mathcal{C}$ are ordinary, i.e., they look locally as $\{x^{\ell} = y^{\ell}\}$ for some $\ell \geq 2$;
    \item there is no point where all the curves meet.
\end{itemize}
\end{definition}

Such a class of point-curve configurations enjoys many algebro-combinatorial properties that are highly desirable in many applications. Let us recall that for $d$-arrangements we have the following combinatorial count
\begin{equation}
d^2 { k \choose 2} = \sum_{r\geq 2} {r \choose 2}t_{r}  \, \,= \sum_{p \in {\rm Sing}(\mathcal{C})} { m_{p} \choose 2},
\end{equation}
where $t_{r}$ denotes the number of $r$-fold points, $m_{p}$ denotes the multiplicity of a given singular point $p \in {\rm Sing}(\mathcal{C})$, and this number is equal to the number $r_{p}$ of analytic branches passing through this point.
Now, if $\mathcal{C}$ is a $d$-arrangement, then we can easily show that restricting to each curve $C_{i} \in \mathcal{C}$ one has
$$d^{2}(k-1) = \sum_{p \in {\rm Sing}(\mathcal{C}) \cap C_{i}}(m_{p}-1),$$
so in particular if on $C_{i}$ the only singular points are double points, then we have exactly $d^{2}(k-1)$ such points. Moreover, we are going to use the following abbreviations:
$$f_{0} = \sum_{r\geq 2}t_{r}, \quad f_{1} = \sum_{r\geq 2}rt_{r} = \sum_{p \in {\rm Sing}(\mathcal{C})}m_{p}.$$

For line arrangements we have the following celebrated inequality. 
\begin{theorem}[Hirzebruch]
Let $\mathcal{L} \subset \mathbb{P}^{2}_{\mathbb{C}}$ be an arrangement of $k\geq 6$ lines such that $t_{k} = t_{k-1} = 0$, then one has $$t_{2} + t_{3} \geq k + \sum_{r\geq 4}(r-4)t_{r}.$$
\end{theorem}
If $d\geq 2$, then for such $d$-arrangements there exists a Hirzebruch-type inequality proved by Pokora, Roulleau, and Szemberg in \cite{PRSz}.
\begin{theorem}\label{PSZR}
Let $\mathcal{C}$ be a $d$-arrangement of $k\geq 3$ curves with $d\geq 2$. Assume that $t_{k}=0$, then
$$\bigg(\frac{7}{2}d - \frac{9}{2}\bigg)dk + t_{2} + t_{3} \geq \sum_{r\geq 4}(r-4)t_{r}.$$
\end{theorem}

 As it was recalled in Introduction, the second author formulated a question about the values of Seshadri constants for point-line arrangements in the complex projective plane. There is no logical constraint to restrict our attention only to line arrangements in the plane since we can also study the multipoint Seshadri constants from the viewpoint of curve arrangements and their combinatorial properties. Our aim here is to generalize Question \ref{MQ} to the class of $d$-arrangements of curves.
\begin{definition}
Let $\mathcal{C} = \{C_{1}, ..., C_{k}\} \subset \mathbb{P}^{2}_{\mathbb{C}}$ be a $d$-arrangement of curves. Then we define the \emph{base constant} of $\mathcal{C}$ as 
$${\rm bs}(\mathcal{C}) := {\rm max} \{s \, | \, s = \# \, C_{i} \cap {\rm Sing}(\mathcal{C}), \, C_{i} \in \mathcal{C}\},$$
i.e., this constant is equal to the maximal number of singular points from ${\rm Sing}(\mathcal{C})$ that is contained in some curve $C_{i} \in \mathcal{C}$.
\end{definition}
\label{eq:ses}
The first, naive, attempt to generalize Question \ref{MQ} to $d$-arrangement could be to ask whether for $\mathcal{C}$ one has
\begin{equation}\varepsilon(\mathbb{P}^{2}_{\mathbb{C}}, \mathcal{O}_{\mathbb{P}^{2}_{\mathbb{C}}}(1); {\rm Sing}(\mathcal{C})) = \frac{1}{{\rm bs}(\mathcal{C})}.
\end{equation}
This works in a number of examples. However, in Section $3$, we show that equality (\ref{eq:ses}) fails in case of the so-called Hesse (or Chilean) arrangement of conics -- \cite{Chilean, SR}. Thus, we put forward the following problem.
\begin{question}\label{QuestM}
Let $\mathcal{C}$ be a $d$-arrangement with $k\geq3$ and $d\geq 1$ having singular locus ${\rm Sing}(\mathcal{C})$. Is it true that
$$\varepsilon(\mathbb{P}^{2}_{\mathbb{C}}, \mathcal{O}_{\mathbb{P}^{2}_{\mathbb{C}}}(1); {\rm Sing}(\mathcal{C})) \geq \frac{1}{d(k-1)}$$
and the equality holds if and only if there is a curves $C_{i} \in \mathcal{C}$ for which ${\rm bs}(\mathcal{C}) = d^{2}(k-1)$?
\end{question}
There is a natural temptation to believe that the lowest possible value for the Seshadri constants can be achieved by $d$-star configurations of curves.
\begin{definition}
We say that an arrangement $\mathcal{C} \subset \mathbb{P}^{2}_{\mathbb{C}}$ of $k\geq3$ curves is called a $d$-star configuration if this is an arrangement of $k$ smooth curves, each of degree $d\geq 1$, in generic position, i.e., the only intersection points are ordinary double points.
\end{definition}
If we consider the case $d=1$, then we have at least two types of line arrangements giving the Seshadri constant equal to $\frac{1}{k-1}$, namely star configurations and Hirzebruch quasi-pencil of lines, i.e., an arrangement of $k\geq 4$ lines such that $t_{k-1}=1$ and $t_{2}=k-1$. This example shows that it might be difficult to have a general classification of point-curve arrangements $\mathcal{C}$ which give the Seshadri constant equal to $\frac{1}{d^{2}(k-1)}$. Observe also that the assumption $t_{k}=0$ is essential in that picture. If we consider the case $k=2$ and $d=2$, then we have a configuration of $4$ double intersection points $\mathcal{P}$ and the Seshadri constant can be computed by a line passing through a pair of two distinct points from $\mathcal{P}$ - this is the case that we want to exclude from our discussion due to triviality.

As a warming-up, we are going to show that $d$-star configurations are good candidates for the actual lower bound in Question \ref{QuestM}.
\begin{proposition}
\label{prop:star}
If $\mathcal{C}_{d} = \{C_{1}, ..., C_{k}\}$ is a $d$-star configuration with $k \geq 3$ and $d\geq 1$, then 
$$\varepsilon(\mathbb{P}^{2}_{\mathbb{C}}, \mathcal{O}_{\mathbb{P}^{2}_{\mathbb{C}}}(1); {\rm Sing}(\mathcal{C})) = \frac{1}{d(k-1)}.$$
\end{proposition}
\begin{proof}
Denote by $\mathcal{P} = \{p_{1}, ..., p_{s}\}$ the set of all double intersection points of $\mathcal{C}_{d}$ and $C = C_{1}+...+C_{k}$. Suppose that there exists an irreducible and reduced curve $D$, different from each $C_{i}$ for $i \in \{1,...,k\}$, having degree $e$ and at each point $p \in \mathcal{P}$ multiplicity $m_{p}(D)$ such that
$$\frac{e}{\sum_{p \in \mathcal{P}} m_{p}(D)} < \frac{1}{d(k-1)}.$$
It means that we have the following bound 
$$(\triangle) : \quad ed(k-1) < \sum_{p \in \mathcal{P}} m_{p}(D).$$
Now we are going to use B\'ezout's theorem, we have
$$edk =D.C = D.(C_{1}+ ... + C_{k}) \geq \sum_{p \in \mathcal{P}}m_{p}(D)\cdot m_{p}(C)\stackrel{(*)}{\geq} 2\sum_{p \in \mathcal{P}} m_{p}(D) \stackrel{(\triangle)}{>} 2ed(k-1),$$
where $(*)$ comes from the fact that all the intersection points of $C$ are double points. This leads to the following inequality:
$$0> 2edk - 2ed -edk = edk - 2ed,$$
a contradiction since $k\geq 3$.
\end{proof}
Our next result tells us that $d$-arrangements are unique in a sense of the associated linear series which means that $\varepsilon_{\mathcal{C}}$ can be viewed as an invariant of a given arrangement $\mathcal{C}$.
\begin{proposition}
Let $\mathcal{C} \subset \mathbb{P}^{2}_{\mathbb{C}}$ be a $d$-arrangement of $k\geq 3$ curves. Consider the blowing up $f: X \rightarrow \mathbb{P}^{2}_{\mathbb{C}}$ at ${\rm Sing}(\mathcal{C})$ with the exceptional divisors $E_{1}, ..., E_{f_{0}}$ and $H = f^{*}(\mathcal{O}_{\mathbb{P}^{2}_{\mathbb{C}}}(1))$. Then the linear series 
\begin{equation}
    \label{linser}
\bigg|{\rm deg}(\mathcal{C})H - \sum_{p \in {\rm Sing}(\mathcal{C})}m_{p}(\mathcal{C})E_{p}\bigg|
\end{equation}
contains exactly one member.
\end{proposition}
\begin{proof}
  Let us denote by $C_{i}'$ the strict transform of $C_{i}$ under the blowing up $f$ for $i \in \{1, ..., k\}$. By the assumption that $t_{k}=0$ we can show that each curve $C_{i}'$ has the self-intersection number less than or equal to $-1$. This follows from the fact that each curve $C_{i}$ contains at least $d^{2}+1$ points. Indeed, if not, then by the assumption that all intersection points are ordinary the curve $C_{i}$ would contain exactly $d^2$ intersection points, but it implies that all $k$ curves $C_{j}$ meet exactly in $d^2$ points, a contradiction. Now we are going to use the uniqueness of the Zariski decomposition of $C' = C_{1}' + ... + C_{k}'$ -- it is an effective cycle such that for $i\neq j$ we have $C_{i}'.C_{j}' = 0$ and $C_{i}' \leq -1$, which means that the intersection matrix of this cycle is negative definite, so we conclude that $C'$ is the unique member in linear series (\ref{linser}).
\end{proof}
\begin{remark}
As it was pointed out by M. Dumnicki, one can relax the assumption that for $d$-arrangements with $d\geq 2$ one should have $t_{k}=0$ -- it is enough to assume that on each curve there is at least one point of multiplicity less than $k$. 
\end{remark}
Our main contribution, from the viewpoint of Question \ref{QuestM}, provides a lower bound on configurational Seshadri constants.
\begin{theorem}
Let $\mathcal{C}$ be a $d$-arrangement of $d\geq 1$ with $k\geq 3$. In the case of $d=1$ we assume additionally that $t_{d-1}=0$. Then 
$$\varepsilon_{\mathcal{C}}(
\mathcal{O}_{\mathbb{P}^{2}_{\mathbb{C}}}(1)) \geq \frac{1}{2dk + 3d/2 -9/2}$$
\end{theorem}
\begin{proof}
Our strategy is based on the combinatorial features of $\mathcal{C}$. Let us denote by $C = C_{1} + ... + C_{k}$. Then we can write 
\begin{equation}\label{ses}
\varepsilon_{\mathcal{C}}(\mathcal{O}_{\mathbb{P}^{2}_{\mathbb{C}}}(1)) = \frac{{\rm deg}(C)}{\sum_{p \in {\rm Sing}(\mathcal{C})}m_{p}} = \frac{dk}{f_{1}}.
\end{equation}
Our goal here is to find a reasonable upper-bound on the number $f_{1} = \sum_{r\geq 2} rt_{r}$. In order to do so, we are going to use Theorem \ref{PSZR} and Hirzebruch's inequality, namely
$$\bigg(\frac{7d}{2}-\frac{9}{2}\bigg)dk - t_{2} \geq \sum_{r\geq 2}(r-4)t_{r} = f_{1} - 4f_{0}.$$
Since $t_{2} \geq 0$ we have
$$\bigg(\frac{7d}{2}-\frac{9}{2}\bigg)dk + 4f_{0} \geq f_{1}.$$
Obviously one always has $$k \leq f_{0} \leq d^{2} {k \choose 2}$$ which leads to
$$f_{1} \leq 2d^{2}k^{2} + \frac{3}{2}d^{2}k - \frac{9}{2}dk,$$
so finally we get
$$\varepsilon_{\mathcal{C}}(\mathcal{O}_{\mathbb{P}^{2}_{\mathbb{C}}}(1)) = \frac{dk}{f_{1}}\geq \frac{dk}{2d^{2}k^{2}+3d^{2}k/2-9dk/2} = \frac{1}{2dk + 3d/2 - 9/2}.$$

\end{proof}
\begin{remark}
The punchline of the above result is that, abusing the $\mathcal{O}$-notation, for $d$-configurations one has
$$\varepsilon_{\mathcal{C}}(
\mathcal{O}_{\mathbb{P}^{2}_{\mathbb{C}}}(1)) = \mathcal{O}\bigg(\frac{1}{dk}\bigg),$$
so we arrive at the predicted order of magnitude.
\end{remark}
\section{The multipoint Seshadri constants via $d$-arrangements}
In this section we are going to present a comparison between configurational Seshadri constants and multipoint Seshadri constants centered at singular loci of $d$-arrangements. It is clear that one always has
$$\varepsilon_{\mathcal{C}}(
\mathcal{O}_{\mathbb{P}^{2}_{\mathbb{C}}}(1)) \geq \varepsilon(\mathbb{P}^{2}_{\mathbb{C}},
\mathcal{O}_{\mathbb{P}^{2}_{\mathbb{C}}}(1); {\rm Sing}(\mathcal{C})).$$ 
First of all, our aim here is to study possible discrepancies between these constants. We start with the so-called large pencils of lines for which the discrepancies are rather significant -- from a viewpoint of computations this stands against our intuition.
\begin{example}
Let $\mathcal{H}$ be a Hirzebruch quasi-pencil, i.e., an arrangement which consists of $k\geq 4$ lines with $t_{k-1}=1$ and $t_{2} = k-1$. It can be easily checked that 
$$\varepsilon(\mathbb{P}^{2}_{\mathbb{C}}, \mathcal{O}_{\mathbb{P}^{2}_{\mathbb{C}}}(1); {\rm Sing}(\mathcal{H})) = \frac{1}{k-1}.$$
On the other side
$$\varepsilon_{\mathcal{H}}(\mathcal{O}_{\mathbb{P}^{2}_{\mathbb{C}}}(1)) = \frac{k}{2\cdot(k-1) + k-1} = \frac{k}{3k-3}$$ which shows that in general we have $\varepsilon_{\mathcal{H}} > \varepsilon$.
\end{example}
\begin{example}
Let us consider arrangements of $k$ lines having $t_{k-2} = 1$. There are two types of such arrangements, namely either
\begin{itemize}
    \item $t_{k-2}=1$ and $t_{2}=2k-3$, or
    \item $t_{k-2}=1$, $t_{3}=1$, $t_{2}=2k-6$.
\end{itemize}
Let us consider the first case (the second one is analogous) and denote the associated arrangement by $\mathcal{HL}$. Among all lines we can take a line passing thought exactly $k-1$ double points. An easy inspection shows that this line computes the Seshadri constant which is equal to $\frac{1}{k-1}$. On the other side,
$$\varepsilon_{\mathcal{HL}}(\mathcal{O}_{\mathbb{P}^{2}_{\mathbb{C}}}(1)) = \frac{k}{2\cdot(2k-3) + k-2} = \frac{k}{5k-8}.$$
\end{example}

Now, we would like to begin our comparison for more complicated arrangements from a viewpoint of combinatorics. We made our comparison with respect to a very interesting class of line arrangements, namely simplicial line arrangements. Let us recall that $\mathcal{L} \subset \mathbb{P}^{2}_{\mathbb{R}}$ is  simplicial if $M(\mathcal{L}) := \mathbb{P}^{2}_{\mathbb{R}} \setminus \bigcup_{H \in \mathcal{L}} H$ the complement of $\mathcal{L}$ in $\mathbb{P}^{2}_{\mathbb{R}}$ consists of the union of disjoint triangles. Our wide computer experiments suggest that the configurational Seshadri constants are more accurate when given line arrangements are \emph{symmetric}, which can be understood both from a viewpoint of the multiplicities of singular points and the symmetry groups of arrangements. From this perspective simplicial arrangements have both mentioned features. Since our computations are rather cumbersome, we decided to present a short table of $11$ simplicial line arrangements for which we provide the actual values of the Seshadri constants and the corresponding values of the configurational Seshadri constants. Here by $\mathcal{A}_{k}(n)$ we denote, according to Cuntz's list \cite{Cuntz}, a simplicial arrangement of $n$ lines of type $k$.
\begin{center}
{\renewcommand{\arraystretch}{1.3}
\begin{tabular}{c|c|c|c}
\hline 
$\mathcal{C}=\mathcal{A}_{k}(n)$ & $ t=(t_2,t_3,t_4,\ldots) $ & $\varepsilon_{\mathcal{C}}$ & $\varepsilon$ \\ 
\hline 
$\mathcal{A}_{1}(6)$ & $(3,4)$ & $\tfrac{1}{3}$ & $\tfrac{1}{3}$ \\ 
\hline 
$\mathcal{A}_{1}(7)$ & $(3,6)$ & $\tfrac{7}{24}$ & $\tfrac{1}{4}$ \\ 
\hline 
$\mathcal{A}_{1}(8)$ & $(4,6,1)$ & $\tfrac{4}{15}$ & $\tfrac{1}{4}$ \\ 
\hline 
$\mathcal{A}_{1}(9)$ & $(6,4,3)$ & $\tfrac{1}{4}$ & $\tfrac{1}{4}$ \\ 
\hline 
$\mathcal{A}_{1}(10)$ & $(5,10,0,1)$ & $\tfrac{2}{9}$ & $\tfrac{1}{5}$ \\ 
\hline 
$\mathcal{A}_{2}(10)$ & $(6,7,3)$ & $\tfrac{2}{9}$ & $\tfrac{1}{6}$ \\ 
\hline 
$\mathcal{A}_{3}(10)$ & $(6,7,3)$ & $\tfrac{2}{9}$ & $\tfrac{1}{5}$ \\
\hline 
$\mathcal{A}_{1}(11)$ & $(7,8,4)$ & $\tfrac{11}{54}$ & $\tfrac{1}{6}$ \\ 
\hline 
$\mathcal{A}_{1}(12)$ & $(6,15,0,0,1)$ & $\tfrac{4}{21}$ & $\tfrac{1}{6}$ \\ 
\hline 
$\mathcal{A}_{2}(12)$ & $(8,10,3,1)$ & $\tfrac{4}{21}$ & $\tfrac{1}{6}$ \\ 
\hline 
$\mathcal{A}_{3}(12)$ & $(9,7,6)$ & $\tfrac{4}{21}$ & $\tfrac{1}{6}$ \\ 
\hline 
\end{tabular} 
}
\end{center}

From now on, we would like to focus on $d$-arrangements. We must emphasize in this point that there are not many examples of such arrangements in literature, and we are going to look at those that are interesting for our considerations. We start with $d$-star configurations.
\begin{example}
Consider $d$-star arrangements $\mathcal{C}_{d}$ with $d\geq 1$ and $k\geq 3$ curves. We have shown in Proposition \ref{prop:star} that one always has
$$\varepsilon(\mathbb{P}^{2}_{\mathbb{C}}, \mathcal{O}_{\mathbb{P}^{2}_{\mathbb{C}}}(1); {\rm Sing}(\mathcal{C}_{d})) = \frac{1}{d(k-1)}.$$ Now we compute the configurational Seshadri constant of $\mathcal{C}_{d}$. We have exactly $t_{2} = d^{2}\cdot \frac{(k^{2}-k)}{2}$ double intersection points, hence
$$\varepsilon_{\mathcal{C}_{d}}(
\mathcal{O}_{\mathbb{P}^{2}_{\mathbb{C}}}(1)) = \frac{dk}{2t_{2}} = \frac{dk}{d^{2}k(k-1)} = \frac{1}{d(k-1)}.$$
\end{example}
\begin{example}
Let us now consider a symmetric $(6_{5}, 6_{5})$-arrangement $\mathcal{PC}$ which is an arrangement of $6$ conics such that $t_{5}=6$. The arrangement is constructed by fixing $6$ points in general position, and then we take all the conics passing through $5$ points from the set of mentioned $6$ points (and in fact it works over the reals). Obviously the Seshadri constant is equal to $\frac{2}{5}$ since $6$ points are not contained in a conic, and we have
$$\varepsilon_{\mathcal{PC}}(
\mathcal{O}_{\mathbb{P}^{2}_{\mathbb{C}}}(1)) = \frac{2}{5}.$$
\end{example}
The next arrangement has been discovered recently by Dolgachev, Laface, Persson, and Urz\'ua in \cite{Chilean}, and also independently by Kohel, Roulleau and Sarti in \cite{SR}.
\begin{example}[The Hesse arrangement of conics]
We would like to take Kohel-Roulleau-Sarti's description. It is well-known that the dual curve to a smooth elliptic curve $\mathcal{E}$ is an irreducible sextic curve having exactly $9$ cusps -- these cusps correspond to the set of $9$ flex points of $\mathcal{E}$. It turns out that there are exactly $12$ irreducible conics such that each conic is passing through a subset of exactly $6$ points of the set of $9$ points determined by the cusps. The resulting point-conic configuration $\mathcal{CL}$ is $(9_{8}, 12_{6})$-configuration, i.e., we have exactly $9$ points of multiplicity $8$, and there are also exactly $12$ nodes. Even more is true, these $12$ nodes are exactly the triple intersection points of the dual Hesse arrangement of $9$ lines. Using a combinatorial count we know that on each conic we have exactly $6$ points of multiplicity $8$ and exactly $2$ nodes. First of all,
$$\varepsilon_{\mathcal{CL}}(
\mathcal{O}_{\mathbb{P}^{2}_{\mathbb{C}}}(1)) = \frac{24}{24 + 72} = \frac{1}{4}.$$
Now we are going to look at potential curves which might compute the Seshadri ratio. Firstly, taking any conic from the arrangement, passing through $8$ singular points, we obtain the ratio equal to $\frac{1}{4}$. On the other side, since the twelve points are the triple intersection points of the dual Hesse arrangement, any line from the dual Hesse arrangement is passing through exactly $4$ points from the set of $12$ nodes and one additional point which turns out to be one of the points of multiplicity $8$. Such a line gives us the ratio equal to $\frac{1}{5}$. Now we are going to show that indeed one has
$$\varepsilon(\mathbb{P}^{2}_{\mathbb{C}},\mathcal{O}_{\mathbb{P}^{2}_{\mathbb{C}}}(1); {\rm Sing}(\mathcal{CL})) = \frac{1}{5}.$$

We will argue in a standard way, but we must use a very specific property of the set of all singular points of the arrangement that is not combinatorial at all. Suppose that there exists an irreducible and reduced plane curve $D$ of degree $e$, different from each line contained in the dual Hesse arrangement of lines passing through the nodes and different from each conic in $\mathcal{CL}$, having the property that
$$\frac{e}{\sum_{i=1}^{21}m_{p_{i}}(D)} < \frac{1}{5},$$
which means that
$$5e < \sum_{i=1}^{21}m_{p_{i}}(D).$$
The position of singular points implies the existence of a very specific curve, namely there exists a plane quintic curve passing through all the $21$ singular points, which can be easily checked with use of Singular script, see Appendix. Let us denote this quintic curve by $Q$. In fact, $Q$ is reducible and it can be taken as a sum of $3$ lines from the dual Hesse arrangement and an irreducible conic from Hesse arrangement of conics. Observe that
$$5e = D.Q \geq \sum_{i=1}^{21}m_{p_{i}}(D)\cdot m_{p_{i}}(Q) \geq \sum_{i=1}^{21}m_{p_{i}}(D) > 5e,$$
a contradiction.
\end{example}

\begin{remark}
The Hesse arrangement of $12$ conics, due to its very specific geometry, should be in fact considered as a conic-line arrangement. If we consider $12$ conics and $9$ lines we get an arrangement $\mathcal{HCL}$ consisting of $21$ curves and having the following intersection points
$$t_{9} = 9, \quad t_{5}=12, \quad t_{2}=72.$$
The arrangement $\mathcal{HCL}$ is described in \cite{Chilean, PSzem}.
\end{remark}
\section*{Acknowledgments}
We would like to warmly thank Marcin Dumnicki and Halszka Tutaj-Gasi\'nska for useful comments that allowed to improve the note, and to Grzegorz Malara for useful discussions. The second author was partially supported by National Science Center (Poland) Sonata Grant Nr \textbf{2018/31/D/ST1/00177}. Finally, we would like to warmly thank an anonymous referee for all useful comments and suggestions.
\section*{Appendix}
Here we present our Singular script that verifies the existence of a plane quintic passing though the singular locus of the Hesse arrangement of $12$ conics. We provide, en passant, the coordinates of all singular points of the Hesse arrangement of conics.
\begin{verbatim}
ring R=(0,u),(x,y,z),dp;
minpoly=31+36*u+27*u2-4*u3+9*u4+u6;
ideal P(1)=90*y+(-4u5+u4-40u3+26u2-92u-91)*z,x-z;
ideal P(2)=36*y+(u5-u4+10u3-20u2+29u-11)*z,10*x+(-u2+2u+11)*y+(-4u2-4u-6)*z;
ideal P(3)=60*y+(u5+u4+10u3+16u2+13u+79)*z,6*x+(u2+2u-11)*y+(-4u2+4u-10)*z;
ideal P(4)=90*y+(u5-4u4+10u3-29u2+53u-11)*z,x-y;
ideal P(5)=60*y+(u5+u4+5u3+u2-2u+44)*z,6*x+(-4u2+4u-10)*y+(u2+2u-11)*z;
ideal P(6)=36*y+(-u5+u4-7u3+11u2-20u-22)*z,10*x+(-4u2-4u-6)*y+(-u2+2u+11)*z;
ideal P(7)=y-z,90*x+(-4u5+u4-40u3+26u2-92u-91)*z;
ideal P(8)=180*y+(4u5-u4+40u3-26u2+182u+181)*z,x+(-u-1)*y+(-u+1)*z;
ideal P(9)=180*y+(-4u5+u4-40u3+26u2-182u-1)*z,x+(-u+1)*y+(-u-1)*z;
ideal P(10)=z,x;
ideal P(11)=y,x;
ideal P(12)=z,y;
ideal P(13)=y-z,x-z;
ideal P(14)=180*y+(4u5-u4+40u3-26u2+182u+181)*z,
180*x+(-4u5+u4-40u3+26u2-182u-1)*z;
ideal P(15)=180*y+(-4u5+u4-40u3+26u2-182u-1)*z,
180*x+(4u5-u4+40u3-26u2+182u+181)*z;
ideal P(16)=180*y+(4u5-u4+40u3-26u2+182u+181)*z,x-z;
ideal P(17)=180*y+(-4u5+u4-40u3+26u2-182u-1)*z,
180*x+(-4u5+u4-40u3+26u2-182u-1)*z;
ideal P(18)=y-z,180*x+(4u5-u4+40u3-26u2+182u+181)*z;
ideal P(19)=180*y+(-4u5+u4-40u3+26u2-182u-1)*z,x-z;
ideal P(20)=y-z,180*x+(-4u5+u4-40u3+26u2-182u-1)*z;
ideal P(21)=180*y+(4u5-u4+40u3-26u2+182u+181)*z,
180*x+(4u5-u4+40u3-26u2+182u+181)*z;
ideal I=1;int i;
for(i=1;i<=21;i++){
I=intersect(I,P(i));
}
I=std(I);
I[1];
\end{verbatim}

\bigskip

\noindent
Marek Janasz,\\
Department of Mathematics, Pedagogical University of Cracow,
Podchor\c a\.zych 2,
PL-30-084 Krak\'ow, Poland.

\nopagebreak
\noindent
\textit{E-mail address:} \texttt{marek.janasz@up.krakow.pl}\\

\bigskip
\noindent
   Piotr Pokora,\\
   Department of Mathematics, Pedagogical University of Cracow,
   Podchor\c a\.zych 2,
   PL-30-084 Krak\'ow, Poland.

\nopagebreak
\noindent
   \textit{E-mail address:} \texttt{piotrpkr@gmail.com, piotr.pokora@up.krakow.pl}\\
\end{document}